\RequirePackage[l2tabu, orthodox]{nag} 

\documentclass[reqno]{amsart}

\pdfoutput=1
\usepackage[utf8]{inputenc} 
\usepackage{bbm} 
\usepackage{commath} 
\usepackage{mathtools} 
\usepackage{todonotes} 
\usepackage[colorlinks, citecolor=red, breaklinks=true]{hyperref} 
\usepackage[stretch=10, shrink=10]{microtype} 
\usepackage[initials,msc-links,non-compressed-cites,nobysame]{amsrefs}[2007/10/22]

\usepackage{xcolor}

\newtheorem{theorem}{Theorem}[section]
\newtheorem{lemma}[theorem]{Lemma}

\newtheorem{fact}[theorem]{Fact}
\newtheorem{corollary}[theorem]{Corollary}

\theoremstyle{definition}
\newtheorem{definition}[theorem]{Definition}

\theoremstyle{remark}
\newtheorem{remark}[theorem]{Remark}
\newtheorem{example}[theorem]{Example}

\numberwithin{equation}{section}

\usepackage{paralist,enumitem}
\setlist{listparindent=0pt,parsep=3pt}

\newcommand{\TitleWithUrl}[1]{\IfEmptyBibField{doi}%
  {\IfEmptyBibField{url}{\textit{#1}}%
    {\IfEmptyBibField{eprint}{\href {\BibField{url}}{\textit{#1}}}{\textit{#1}}}%
    }%
  {\href {https://doi.org/\BibField{doi}}{\textit{#1}}}}
\renewcommand{\eprint}[1]{\IfEmptyBibField{url}{\url{#1}}%
  {\href {\BibField{url}}{#1}}}

\BibSpec{article}{%
    +{}  {\PrintAuthors}                {author}
    +{,} { \TitleWithUrl}               {title}
    +{.} { }                            {part}
    +{:} { \textit}                     {subtitle}
    +{,} { \PrintContributions}         {contribution}
    +{.} { \PrintPartials}              {partial}
    +{,} { }                            {journal}
    +{}  { \textbf}                     {volume}
    +{}  { \PrintDatePV}                {date}
    +{,} { \issuetext}                  {number}
    +{,} { \eprintpages}                {pages}
    +{,} { }                            {status}
    +{,} { available at arXiv:\eprint}        {eprint}
    +{}  { \parenthesize}               {language}
    +{}  { \PrintTranslation}           {translation}
    +{;} { \PrintReprint}               {reprint}
    +{.} { }                            {note}
    +{.} {}                             {transition}
    +{}  {\SentenceSpace \PrintReviews} {review}
}

\BibSpec{collection.article}{%
    +{}  {\PrintAuthors}                {author}
    +{,} { \TitleWithUrl}                     {title}
    +{.} { }                            {part}
    +{:} { \textit}                     {subtitle}
    +{,} { \PrintContributions}         {contribution}
    +{,} { \PrintConference}            {conference}
    +{}  {\PrintBook}                   {book}
    +{,} { }                            {booktitle}
    +{,} { \PrintDateB}                 {date}
    +{,} { pp.~}                        {pages}
    +{,} { }                            {status}
    +{,} { available at \eprint}        {eprint}
    +{}  { \parenthesize}               {language}
    +{}  { \PrintTranslation}           {translation}
    +{;} { \PrintReprint}               {reprint}
    +{.} { }                            {note}
    +{.} {}                             {transition}
    +{}  {\SentenceSpace \PrintReviews} {review}
}
\BibSpec{book}{%
    +{}  {\PrintPrimary}                {transition}
    +{,} { \TitleWithUrl}               {title}
    +{.} { }                            {part}
    +{:} { \textit}                     {subtitle}
    +{,} { \PrintEdition}               {edition}
    +{}  { \PrintEditorsB}              {editor}
    +{,} { \PrintTranslatorsC}          {translator}
    +{,} { \PrintContributions}         {contribution}
    +{,} { }                            {series}
    +{,} { \voltext}                    {volume}
    +{,} { }                            {publisher}
    +{,} { }                            {organization}
    +{,} { }                            {address}
    +{,} { \PrintDateB}                 {date}
    +{,} { }                            {status}
    +{}  { \parenthesize}               {language}
    +{}  { \PrintTranslation}           {translation}
    +{;} { \PrintReprint}               {reprint}
    +{.} { }                            {note}
    +{.} {}                             {transition}
    +{}  {\SentenceSpace \PrintReviews} {review}
}

\BibSpec{thesis}{%
    +{}  {\PrintAuthors}                {author}
    +{.} { \TitleWithUrl}               {title}
    +{:} { \textit}                     {subtitle}
    +{,} { \PrintThesisType}            {type}
    +{,} { }                            {organization}
    +{} { \PrintDate}                  {date}
    +{,} { }                            {address}
    +{,} { \eprint}                     {eprint}
    +{,} { }                            {status}
    +{}  { \parenthesize}               {language}
    +{}  { \PrintTranslation}           {translation}
    +{;} { \PrintReprint}               {reprint}
    +{.} { }                            {note}
    +{.} {}                             {transition}
    +{}  {\SentenceSpace \PrintReviews} {review}
}

\begin{document}

\title{Monodromy of Darboux transformations of polarised curves}

\author{Joseph Cho}
\address[Joseph Cho]{Global Leadership School, Handong Global University, 558 Handong-ro Buk-gu, Pohang, Gyeongsangbuk-do, 37554, Republic of Korea}
\email{jcho@handong.edu}

\author{Katrin Leschke}
\address[Katrin Leschke]{School of Computing and Mathematical Sciences, University of Leicester, University Road, Leicester LE1 7RH, United Kingdom}
\email{k.leschke@leicester.ac.uk }

\author{Yuta Ogata}
\address[Yuta Ogata]{Department of Mathematics, Faculty of Science, Kyoto Sangyo University, Motoyama, Kamigamo, Kita-ku, Kyoto-City, 603-8555, Japan}
\email{yogata@cc.kyoto-su.ac.jp}

\subjclass[2020]{Primary: 53A04, Secondary: 53A31, 58J72}
\keywords{polarised curve, monodromy, Darboux transformation, resonance point}

\begin{abstract}
	We show that every finite type polarised curve in the conformal $2$-sphere with a polynomial conserved quantity admits a resonance point, under a non-orthogonality assumption on the conserved quantity.
	Using this fact, we deduce that every finite type curve polarised by space form arc-length in the conformal $2$-sphere admits a resonance point, possibly on a multiple cover.
\end{abstract}

\maketitle

\section{Introduction}
In surface theory, examples with global properties such as constant curvature, symmetry, completeness, embeddedness, compactness, non-trivial topology, are highly sought after, with emphasis being given to those examples that exhibit multiple of these properties.
Two of the perhaps most well-known examples constructed in modern surface theory, the Costa--Hoffman--Meeks surface \cite{costa_ExampleCompleteMinimal_1984, hoffman_CompleteEmbeddedMinimal_1985, hoffman_EmbeddedMinimalSurfaces_1990} and the Wente tori \cite{wente_CounterexampleConjectureHopf_1986}, both satisfy multiple of such global properties: the former being globally embedded complete minimal surfaces with non-trivial topology, while the latter are compact constant mean curvature surfaces with the topology of a torus.

Transformation theory has also been a useful tool in discovering examples with multiple global properties.
For example, constant mean curvature (cmc) cylinders of finite type known as bubbletons \cite{sterling_ExistenceClassificationConstant_1993} are found from the circular cylinder, while Bernstein tori are found via from the homogeneous tori in the $3$-sphere \cite{bernstein_NonspecialNoncanalIsothermic_2001}, where both examples arise from the use of integrable transformation theory of surfaces.
In particular, both surfaces are well-known examples of \emph{isothermic surfaces}, an integrable class of surfaces \cite{cieslinski_IsothermicSurfaces$mathbf_1995} comprised of surfaces that admit conformal curvature line coordinates away from umbilics, and the transformations utilized for their construction can be interpreted as Darboux transformations of isothermic surfaces \cite{darboux_SurfacesIsothermiques_1899}.
As the subject of isothermic tori and the methods of their construction has received renewed interest due to its central role in the resolution of the global Bonnet problem \cite{bobenko_CompactBonnetPairs_2021}, one can expect Darboux transformations to become a useful method for the construction of various isothermic tori.

The integrable structure and transformation theory of isothemic surfaces is closely tied to those of \emph{polarised curves} \cite{burstall_SemidiscreteIsothermicSurfaces_2016}.
Several explicit examples of cmc cylinders and isothermic cylinders \cite{cho_GeneralisedBianchiPermutability_2022, cho_NewExplicitCMC_2022} are constructed via Darboux transformations from surfaces of revolutions; in both examples, the transformation theory of polarised curves played an implicit yet central role.
The key ingredient in the construction of these examples is the existence of a \emph{resonance point} for the circular curvature lines viewed as polarised curves.

Typically a Darboux transformation is locally defined via a Riccati-type equation involving choices of a spectral parameter and an initial condition, and the closure of the Darboux transform depends on these choices.
However, for a special spectral parameter, a so-called resonance point, the Darboux transform remains closed regardless of the choice of the initial condition.
In the case of a circle, the Riccati-type equation can be solved explicitly, so that the existence of the resonance point can be checked directly.
The use of this resonance point as the spectral parameter then leads to the explicit construction of the aforementioned cmc cylinders and isothermic cylinders via Darboux transformations.

These examples suggest that the reduction of the integrable structure of isothermic surfaces to those of polarised curves could be a powerful tool in the construction of isothermic tori by way of Darboux transformations.
However, this method remains largely unexplored, mainly due to the lack of known closed polarised curves that admit resonance points.

In this paper, we show that any closed curve in the conformal $2$-sphere\footnote{This includes all $2$-dimensional Riemannian space forms with constant sectional curvature, such as Euclidean plane, hyperbolic plane, or the $2$-sphere.} with a certain polarisation admits resonance points, under the assumption that they are of \emph{finite type} (see, for example, \cite{hitchin_HarmonicMaps$2$torus_1990, pinkall_ClassificationConstantMean_1989}).

To achieve the goal, we first review in Section~\ref{sect:conf} the integrable theory of polarised curves in the conformal $2$-sphere, centered around a $1$-parameter family of flat connections \cite{burstall_SemidiscreteIsothermicSurfaces_2016}.
Then we reiterate the integrable hierarchy of polarised curves in Section~\ref{sect:pcq} given by the order of \emph{polynomial conserved quantities} \cite{burstall_SpecialIsothermicSurfaces_2012} (see also \cite{cho_ConstrainedElasticCurves_2023}).
These are parallel sections of the family of flat connections depending affine polynomially on the parameter.

Polynomial conserved quantities play a key role in Chapter~\ref{sect:monopcq}, where we prove in Theorem~\ref{thm:respcq} that every closed finite type polarised curve with a polynomial conserved quantity admits a resonance point, subject to a non-orthogonality condition imposed on the polynomial conserved quantity.
Finally, in Chapter~\ref{sect:monolcq}, we show in Theorem~\ref{thm:lcq} that every finite type curve polarised by arc-length with respect to a space form metric admits a linear conserved quantity.
This leads us to conclude in Theorem~\ref{thm:resonanceLCQ} that every curve polarised by arc-length with respect to a space form metric admits a resonance points on a multiple cover.
Our result applies to all finite type \emph{constrained elastic curve}, which are curves that are critical points of the bending energy under variations that fix length and enclosed area (Corollary~\ref{coro:cec}).

\textbf{Acknowledgements.}
The authors would like to thank Fran Burstall for useful discussions involving the subject.
We gratefully acknowledge the partial support from: London Mathematical Society (Research in Pairs Scheme Ref:42254); Research Institute for Mathematical Sciences, an International Joint Usage/Research Center located in Kyoto University (Workshop (Type B)); MEXT Promotion of Distinctive Joint Research Center Program JPMXP0723833165 (Joint Usage/Research (C) and International Joint Research (W) at Osaka Central Advanced Mathematical Institute); Japan Society for the Promotion of Science (Grant-in-Aid for Scientific Research (C) 24K06722); and Handong Global University (New Faculty Grant).
Additionally, the second author gratefully acknowledges that the work was finalised during a study leave granted by the University of Leicester.

\section{Preliminaries}
In this section, we review the integrable theory of polarised curves in conformal geometry \cite{burstall_SemidiscreteIsothermicSurfaces_2016}.
\subsection{Lightcone model of conformal geometry}\label{sect:conf}
Let $\mathbb{R}^{3,1}$ denote the pseudo-Euclidean space equipped with symmetric bilinear form $(\cdot, \cdot)$ with signature $(- + + +{})$, and let
    \[
        \mathcal{L} := \{ X \in \mathbb{R}^{3,1}: (X,X) = 0\}
    \]
be the lightcone of $\mathbb{R}^{3,1}$.
For some $\mathfrak{q} \in \mathbb{R}^{3,1}$ with $(\mathfrak{q}, \mathfrak{q}) = -\kappa$, if $M$ is given by
    \[
        M := \{X \in \mathcal{L}: (X, \mathfrak{q}) = -1\},
    \]
then $M$ is a $2$-dimensional Riemannian space form with constant sectional curvature $\kappa$.
For this reason, $\mathfrak{q}$ is referred to as \emph{space form vector}.
We will explain this in detail to set the notations to be used throughout.

\subsubsection{Euclidean space form}
Let us first choose a space form vector $\check{\mathfrak{q}} \in \mathcal{L}$, and consider
    \[
        \check{M} := \{X \in \mathcal{L}: (X, \check{\mathfrak{q}}) = -1\}.
    \]
Then fixing some $\mathfrak{o} \in \mathcal{L}$ such that $(\mathfrak{o}, \check{\mathfrak{q}}) = -1$ gives us
    \[
        \langle \mathfrak{o}, \check{\mathfrak{q}} \rangle \cong \mathbb{R}^{1,1},
    \]
so that
    \[
         \check{\mathfrak{R}} := \langle \mathfrak{o}, \check{\mathfrak{q}} \rangle^\perp \cong \mathbb{R}^2.
    \]
For any $x \in \mathfrak{R}$, we will write
    \[
        |x|^2 := (x,x).
    \]
Under this idenfication, we see that $\check{\psi} : \check{\mathfrak{R}} \to \check{M}$ defined by
    \[
        \check{X} := \check{\psi}(x) = x + \mathfrak{o} + \frac{1}{2}|x|^2 \check{\mathfrak{q}}
    \]
is a bijection.
Thus, we take $\check{\psi}$ to be the coordinate chart of $\check{M}$.

In particular, for any curve $x: (-\epsilon, \epsilon) \to \check{\mathfrak{R}}$ with $\check{X} = \check{\psi}(x)$, we can verify
    \[
        \check{X}' = x' + (x', x) \check{\mathfrak{q}},
    \]
so that
    \begin{equation}\label{eqn:0metric}
        (\check{X}', \check{X}') = |x'|^2.
    \end{equation}
Thus, $\check{M}$ is isometric to the usual Euclidean plane $\mathbb{R}^2$, and the map $\check{\phi} := \check{\psi}^{-1} : \check{M} \to \check{\mathfrak{R}}$ is referred to as the \emph{Euclidean space form projection}.

\subsubsection{Riemannian space forms with constant sectional curvature}
Now for any fixed $\kappa \in \mathbb{R}$, let us define the space form vector as
    \begin{equation}\label{eqn:qkappa}
        \mathfrak{q}_\kappa := \frac{1}{2}(\check{\mathfrak{q}} + 2 \kappa \mathfrak{o}),
    \end{equation}
so that
    \[
        (\mathfrak{q}_\kappa,\mathfrak{q}_\kappa) = \frac{1}{4}(4 \kappa (\check{\mathfrak{q}}, \mathfrak{o})) = -\kappa.
    \]
Defining
    \[
         M_\kappa := \{X \in \mathcal{L}: (X, \mathfrak{q}_\kappa) = -1\},
    \]
and
    \[
        \mathfrak{R}_\kappa := \{x \in \mathbb{R}^2 : |x|^2 \neq -\kappa^{-1}\},
    \]
we would like to find a local coordinate chart $\psi_\kappa : \mathfrak{R}_\kappa \to M_\kappa$.

For this, let $X_\kappa = \alpha_\kappa \check{X}$ for some $\check{X} \in \check{M}$ and $\alpha_\kappa \in \mathbb{R}$ for $X_\kappa \in M_\kappa$, and the condition that $(X_\kappa, \mathfrak{q}_\kappa) = -1$ allows us the calculate that
    \[
        \alpha_\kappa = \frac{2}{1 + \kappa |x|^2}.
    \]
Thus,
    \[
        \psi_\kappa(x) := \alpha_\kappa \psi(x) = \alpha_\kappa \check{X}
    \]
gives a local coordinate chart of $M_\kappa$.

Now, if $X_\kappa : (-\epsilon, \epsilon) \to M_\kappa$ is a curve, then
    \[
        X_\kappa' = \alpha_\kappa' \check{X} + \alpha_\kappa \check{X}',
    \]
and
    \begin{equation}\label{eqn:kappametric}
        (X_\kappa', X_\kappa') = \alpha^2 (\check{X}', \check{X}') =  \frac{4}{(1 + \kappa |x|)^2} |x'|^2.
    \end{equation}
Therefore, $M_\kappa$ (and thus $\mathfrak{R}_\kappa$) is a Riemannian space form with constant sectional curvature $\kappa$ (where $\mathfrak{R}_\kappa$ is the stereoprojection of the curved spaces onto the $2$-plane).
Denoting by $\phi_\kappa := \psi_\kappa^{-1} : M_\kappa \to \mathfrak{R}_\kappa$, we will call $x := \phi_\kappa(X_\kappa)$ a \emph{space form projection}.

\subsubsection{Conformal $2$-sphere and M\"obius transformations} $M_\kappa$ for different values of $\kappa$ are all conformally equivalent, and can be identified under stereoprojections; thus, the projective lightcone
    \[
        \mathbb{P}(\mathcal{L}) := \{ L = \langle X \rangle : X \in \mathcal{L} \}
    \]
represents the \emph{conformal $2$-sphere}.

Under this setting, the set of orthogonal transformations $\mathrm{O}(3,1)$ acts on the conformal $2$-sphere as M\"obius transformations.
Throughout the notes, we will identify $\wedge^2 \mathbb{R}^{3,1}$ with $\mathfrak{o}(3,1)$ where for $x, y \in \mathbb{R}^{3,1}$, $x \wedge y \in \mathfrak{o}(3,1)$ is evaluated via
    \[
        (x \wedge y) v = (x,v) y - (y, v) x
    \]
for any $v \in \mathbb{R}^{3,1}$.

\begin{example}[Poincar\'{e} disk model of hyperbolic plane]\label{exam:disk}
    If we normalize so that $\check{\mathfrak{q}} = (1,0,0, -1)^t$ with $\mathfrak{o} = \tfrac{1}{2}(1,0,0,1)^t$, then for any $x = (0, x_1, x_2, 0)^t \in  \langle \mathfrak{o}, \check{\mathfrak{q}} \rangle^\perp \cong \mathbb{R}^2$, we can verify
        \begin{equation}\label{eqn:x0param}
            \check{X} = \check{\psi}(x) = \begin{pmatrix}
                \tfrac{1}{2}(1 + |x|^2))\\ x_1 \\ x_2 \\ \tfrac{1}{2}(1 - |x|^2)
            \end{pmatrix}.
        \end{equation}
        
    Choosing $\kappa = -1$ gives us
        \[
            \alpha_{-1} = \frac{2}{1 - |x|^2},
        \]
    and
        \[
            X_{-1} = \alpha_{-1} \check{X} = \frac{2}{1 - |x|^2}\begin{pmatrix}
                \tfrac{1}{2}(1 + |x|^2)\\ x_1 \\ x_2 \\ \tfrac{1}{2}(1 - |x|^2).
            \end{pmatrix} = \begin{pmatrix}
               \frac{1 + |x|^2}{1 - |x|^2}\\ \frac{2 x_1}{1 - |x|^2} \\ \frac{2 x_2}{1 - |x|^2} \\ 1
            \end{pmatrix}.
        \]
    Therefore, $M_{-1}$ gives the Poincar\'{e} disk model of hyperbolic plane when $|x|^2 < 1$, and is comprised of two copies of the hyperbolic plane $\mathbb{H}^2$.
\end{example}
\begin{example}[Poincar\'{e} half-plane model of hyperbolic plane]\label{exam:halfplane}
    In addition to the choices of $\check{\mathfrak{q}}$ and $\mathfrak{o}$ in Example~\ref{exam:disk}, let us choose $\tilde{\mathfrak{q}} := (0,0,-1,0)^t$ so that $(\tilde{\mathfrak{q}},\check{\mathfrak{q}}) = 0$ and $(\tilde{\mathfrak{q}}, \tilde{\mathfrak{q}}) = 1$, and define
        \[
            \tilde{M} := \{ X \in \mathcal{L} : (X, \tilde{\mathfrak{q}}) = -1\}.
        \]
    Defining $\tilde{\alpha}$ via $\tilde{\alpha}\check{X} := \tilde{X} \in \tilde{M}$, we have using \eqref{eqn:x0param}
        \[
            -1 = (\tilde{X}, \tilde{\mathfrak{q}}) = (\tilde{\alpha} \check{X}, \tilde{\mathfrak{q}})
            = -\tilde{\alpha} x_2.
        \]
    Hence, $\tilde{\psi} : \tilde{\mathfrak{R}} \to \tilde{M}$ is a bijection with $\tilde{\phi} := \tilde{\psi}^{-1} : \tilde{M} \to \tilde{\mathfrak{R}} := \{x = (x_1, x_2) \in \mathbb{R}^2 : x_2 \neq 0 \}$ where
        \[
            \tilde{\psi}(x) = \tilde{\psi}(x_1, x_2) = \tilde{X} = \tilde{\alpha} \check{X} = \frac{1}{x_2} \check{X} = \begin{pmatrix}
                \tfrac{1}{2 x_2}(1 + |x|^2)\\ \frac{x_1}{x_2} \\ 1 \\ \tfrac{1}{2x_2}(1 - |x|^2)
            \end{pmatrix}.
        \]
    Then since
        \[
            \tilde{X}' = \tilde{\alpha}'\check{X} + \tilde{\alpha} \check{X}',
        \]
    we have
        \[
            (\tilde{X}', \tilde{X}') = (\tilde{\alpha}'\check{X} + \tilde{\alpha} \check{X}', \tilde{\alpha}'\check{X} + \tilde{\alpha} \check{X}') = \tilde{\alpha}^2 (\check{X}', \check{X}') = \frac{|x'|^2}{x_2^2}.
        \]
    Therefore, $\tilde{M}$ gives the Poincar\'{e} half plane model when $x_2 > 0$, and is again comprised of two copies of the hyperbolic plane.
    In particular, we see that the disk model and half-plane model of hyperbolic plane are related by a M\"{o}bius transformation $A \in \mathrm{O}(3,1)$ such that $A \mathfrak{q}_{-1} = \tilde{\mathfrak{q}}$.
\end{example}

\subsection{Polarised curves with polynomial conserved quantities}\label{sect:pcq}
We now review the integrable theory of polarised curves within the lightcone model of conformal geometry as set forth in \cite{burstall_SemidiscreteIsothermicSurfaces_2016}.
Let $I$ denote some open interval in the reals, and consider regular curves defined on a polarised domain $x : (I, q) \to \mathbb{R}^2$, referred to as \emph{polarised curves} for some non-vanishing quadratic differential $q$.
When a coordinate $s \in I$ of the curve $s \in I$ is chosen, then the quadratic differential can be written as
	\[
		q = \frac{1}{m}\dif{s}^2
	\]
where $m$ is a non-vanishing real-valued function on $I$.

As the notion of polarised curves is conformally invariant, we can lift a polarised curve $x$ into the conformal $2$-sphere $L : (I, q) \to \mathbb{P}(\mathcal{L})$ via $L = \langle \check{X} \rangle$, viewed as a subbundle of the trivial bundle $\underline{\mathbb{R}}^{3,1} := I \times \mathbb{R}^{3,1}$.
Referring to $\check{X} \in \Gamma L$ such that $\check{X} : (I, q) \to \check{M}$ as the \emph{Euclidean lift}, the one-parameter family of associated flat connections \cite[Equation~(2.10)]{burstall_SemidiscreteIsothermicSurfaces_2016} of a polarised curve is
    \[
        \dif{}^t := \dif{} + t \eta
    \]
where $\eta \in \Omega^1(L \wedge L^\perp)$ given by
    \[
        \eta := q \frac{\check{X} \wedge \dif{\check{X}}}{(\dif{\check{X}}, \dif{\check{X}})}.
    \]
For any other $X \in \Gamma L$ so that $\check{X} = \alpha X$, we have
    \[
        \eta = q \frac{\alpha X \wedge (\dif{\alpha} X + \alpha \dif{X})}{(\alpha \dif{X}, \alpha \dif{X})}
            = q\frac{\alpha X \wedge \alpha \dif{X}}{(\alpha \dif{X}, \alpha \dif{X})}
            = q \frac{X \wedge \dif{X}}{(\dif{X}, \dif{X})},
    \]
telling us that $\eta$ is independent of the choice of section $X \in \Gamma L$.

For any two parallel sections $\sigma_1, \sigma_2$ of $\dif{}^t$, we have the following lemma:
\begin{lemma}\label{lem:par}
    Let $\sigma_1, \sigma_2$ be parallel sections of the associated family of flat connections $\dif{}^t$.
    Then $(\sigma_1,\sigma_2)$ is constant.
\end{lemma}
\begin{proof}
    Since $\sigma_1, \sigma_2$ are $\dif{}^t$-parallel, we can calculate that
        \[
            \dif{(\sigma_1,\sigma_2)} = (\dif{\sigma_1}, \sigma_2) + (\sigma_1, \dif{\sigma_2}) = - (t \eta \sigma_1, \sigma_2) - (\sigma_1, t \eta \sigma_2) = 0
        \]
    since $\eta$ is skew-symmetric.
\end{proof}

Now we recall the definition of Darboux transformations of polarised curves \cite[Definition~2.5]{burstall_SemidiscreteIsothermicSurfaces_2016}:
\begin{definition}
    Let $L : (I,q) \to \mathbb{P}(\mathcal{L})$ be a polarized curve with associated family of flat connections $\dif{}^t$.
    If $\hat{L} : (I,q) \to \mathbb{P}(\mathcal{L})$ is a parallel section of $\dif{}^\mu$ for some non-zero constant $\mu \in \mathbb{R}$, that is, there is some $\hat{X} \in \Gamma \hat{L}$ such that
        \[
            \dif{}^\mu \hat{X} = (\dif{} + \mu \eta) \hat{X} = 0,
        \]
    we say that $\hat{L}$ is a \emph{Darboux transform} of $L$ with spectral parameter $\mu$, or that $L$ and $\hat{L}$ are a \emph{Darboux pair} with spectral parameter $\mu$.
\end{definition}

To see how the associated family of flat connections of Darboux transforms are given in terms of those of the given curve, consider the splitting of $\mathbb{R}^{3,1}$ via
    \begin{equation}\label{eqn:splitting}
        \mathbb{R}^{3,1} = L_1 \oplus (L_1 \oplus L_2)^\perp \oplus L_2 =: L_1 \oplus W \oplus L_2
    \end{equation}
for some distinct $L_1, L_2 \in \mathbb{P}(\mathcal{L})$.
Defining an orthogonal transformation $\Gamma_{L_1}^{L_2}(r) \in \mathrm{O}(3,1)$ via
    \[
        \Gamma_{L_1}^{L_2}(r)X =
            \begin{cases}
                rX, &\text{if }X \in L_2 \\
                X, &\text{if }X \in W \\
                \frac{1}{r}X, &\text{if }X \in L_1
            \end{cases}
    \]
for some real constant $r \in \mathbb{R}$, the associated family of flat connections of Darboux transforms $\hat{L}$ can be found as follows \cite[Lemma~2.6]{burstall_SemidiscreteIsothermicSurfaces_2016}:
\begin{fact}\label{fact:gauge}
    Let  $L, \hat{L} : (I, q) \to \mathbb{P}(\mathcal{L})$ be a Darboux pair with respect to $\mu$.
    If $\dif{}^t$ and $\hat{\dif{}}^t$ are the associated family of flat connections of $L$ and $\hat{L}$ respectively, then they are related by a gauge transformation
        \[
            \hat{\dif{}}^t = \Gamma_L^{\hat{L}} (1 - \tfrac{t}{\mu}) \bullet \dif{}^t.
        \]
\end{fact}

Also recall the definition of polynomial conserved quantities of a polarised curve \cite[Definition~3.5]{cho_ConstrainedElasticCurves_2023} (see also \cite[Definition~2.1]{burstall_SpecialIsothermicSurfaces_2012}):
\begin{definition}
    Let $L : (I,q) \to \mathbb{P}(\mathcal{L})$ be a polarized curve with associated family of flat connections $\dif{}^t$.
    We say that
        \[
            p^t := p_0 + p_1 t + \ldots + p_d t^d \in (\Gamma \underline{\mathbb{R}}^{3,1})[t]
        \]
    is a \emph{polynomial conserved quantity} of $L$ (or $\dif{}^t$) if $p^t$ is $\dif{}^t$-parallel.
\end{definition}

A straight-forward calculation comparing the coefficients shows the following (cf.\ \cite[Proposition~2.2]{burstall_SpecialIsothermicSurfaces_2012}, \cite[Lemma~3.6]{cho_ConstrainedElasticCurves_2023}):
\begin{lemma}\label{lemma:pcqcond}
    A polynomial in $t$ given by $p^t = \sum_{n=0}^d p_n t^n$ is a polynomial conserved quantity of $L$ with associated family of flat connections $\dif{}^t = \dif{} + t\eta$ if and only if
        \begin{enumerate}
            \item $p_0$ is constant,
            \item $\dif{p_n} + \eta p_{n-1} = 0$ for all $n \in \{1, \ldots, d\}$, and
            \item $p_d$ is perpendicular to both $X$ and $\dif{X}$ for any section $X \in \Gamma L$.
        \end{enumerate}
\end{lemma}

Suppose that $p^t$ is a polynomial conserved quantity of $L$ so that
    \[
        \dif{}^t p^t = (\dif{} + t \eta) p^t = 0.
    \]
Denoting by $\hat{L}$ the Darboux transform of $L$ with spectral parameter $\mu$, if we define
    \[
        \hat{p}^t := (1 - \tfrac{t}{\mu}) \Gamma_L^{\hat{L}} (1 - \tfrac{t}{\mu}) p^t,
    \]
then by Fact~\ref{fact:gauge}, we have
    \[
        \hat{\dif{}}^t \hat{p}^t = (1 - \tfrac{t}{\mu}) \Gamma_L^{\hat{L}} (1 - \tfrac{t}{\mu}) \dif{}^t \left( \Gamma_L^{\hat{L}} (1 - \tfrac{t}{\mu})\right)^{-1}  \Gamma_L^{\hat{L}} (1 - \tfrac{t}{\mu}) p^t = 0,
    \]
so that $\hat{p}^t$ is a polynomial conserved quantity of $\hat{L}$.

In particular, using the splitting \eqref{eqn:splitting} of $\underline{\mathbb{R}}^{3,1}$ with $L$ and $\hat{L}$, we may write
    \begin{equation}\label{eqn:ptsplit}
        p^t = [p^t]_L + [p^t]_W + [p^t]_{\hat{L}},
    \end{equation}
and
    \[
        \hat{p}^t = [p^t]_L + (1 - \tfrac{t}{\mu}) [p^t]_W + (1 - \tfrac{t}{\mu})^2 [p^t]_{\hat{L}}.
    \]
Now, writing $p_d = [p_d]_L + [p_d]_W + [p_d]_{\hat{L}}$, Lemma~\ref{lemma:pcqcond} implies
    \[
        0 = (p_d, X) = ([p_d]_{\hat{L}}, X)
    \]
so that $[p_d]_{\hat{L}} = 0$, allowing us to observe that
    \[
        [p^t]_{\hat{L}} = \sum_{n=0}^{d-1} [p_n]_{\hat{L}} t^n
    \]
is a polynomial of degree at most $d-1$.

Thus it follows as in \cite[Theorem~3.2]{burstall_SpecialIsothermicSurfaces_2012}:
\begin{theorem}\label{thm:pcqp1}
    Let $L : (I,q) \to \mathbb{P}(\mathcal{L})$ be a polarized curve with associated family of flat connections $\dif{}^t$ with polynomial conserved quantity $p^t$ of order $d$.
    If $\hat{L}$ is a Darboux transform of $L$, then $\hat{L}$ admits a polynomial conserved quantity of order at most $d+1$.
\end{theorem}

Now, let $\hat{L}$ be a Darboux transform of $L$ with spectral parameter $\mu$, where $L$ admits a polynomial conserved quantity $p^t$ of order $d$, and further assume that
    \[
        (p^\mu, \hat{X}) = 0
    \]
for $\hat{X} \in \Gamma \hat{L}$.
Then using \eqref{eqn:ptsplit},
    \[
        0 = (p^\mu, \hat{X}) = ([p^\mu]_L, \hat{X}),
    \]
and since $[p^\mu]_L \in \Gamma L$, it follows that
    \[
        [p^\mu]_L = 0.
    \]
Therefore,
    \[
        [p^t]_L = (t-\mu) Q^t
    \]
for some $Q^t \in \Gamma(\underline{\mathbb{R}}^{3,1})[t]$, so that
    \[
        \tilde{p}^t := \Gamma_L^{\hat{L}} (1 - \tfrac{t}{\mu}) p^t
    \]
is a polynomial conserved quantity of $\hat{L}$ of order $d$.
Finally, using Lemma~\ref{lem:par}, it follows as in \cite[Theorem~3.1]{burstall_SpecialIsothermicSurfaces_2012}:
\begin{theorem}\label{thm:backlund}
    Let $L : (I,q) \to \mathbb{P}(\mathcal{L})$ be a polarized curve with associated family of flat connections $\dif{}^t$ with polynomial conserved quantity $p^t$ of order $d$.
    If $\hat{L}$ is a Darboux transform of $L$, such that $\hat{L} \perp p^\mu$ at one point, then $\hat{L}$ admits a polynomial conserved quantity of order $d$.
    Such $\hat{L}$ is referred to as a \emph{B\"{a}cklund-type Darboux transform} of $L$.
\end{theorem}

\section{Monodromy of Darboux transforms of a closed curve with polynomial conserved quantities}\label{sect:monopcq}
Now assume that $L$ is a closed curve with period $T$.
To investigate the global properties of the curves its Darboux transformations, we view all polarised curves $L$ in the conformal $2$-sphere as a subbundle of the trivial bundle defined on the universal cover; however, we will abuse notation and identify the universal cover of $I$ with the interval $I$ throughout.

To investigate the monodromy of Darboux transforms, we choose some coordinates $s \in I$, and let $A^\mu : I \to \mathrm{SO}(3,1)$ be the unique solution given by the ordinary differential equation
    \[
        \frac{\dif{}}{\dif{s}} A^\mu = - \mu \eta A^\mu
    \]
with the initial condition $A^\mu(0) = id$ for any fixed $\mu \in \mathbb{R}$.
Under this setting, $\mathcal{M}^\mu := A^\mu(T)$ is called the \emph{monodromy matrix} of Darboux transforms of $L$ with spectral parameter $\mu$, and is real-analytic in $\mu \in \mathbb{R}$.
Thus, any Darboux transform $\hat{L} = \langle \hat{X} \rangle$ with spectral parameter $\mu$ can be found via
	\[
		\hat{X} = A^\mu \hat{X}_0
	\]
for any choice of $\hat{X}_0 \in \mathcal{L}$.
If, in addition, $\hat{X}_0$ is an eigenvector of $\mathcal{M}^\mu$, then $\hat{L}$ will become a closed Darboux transform with period $T$.

Assume throughout that the curve is of \emph{finite type}  (see, for example, \cite{hitchin_HarmonicMaps$2$torus_1990, pinkall_ClassificationConstantMean_1989}) so that the set 
    \[
        \{ \mu \in \mathbb{R} : \mathcal{M}^\mu \in \mathrm{SO}(3,1)\text{ is non-diagonalisable}\}
    \]
consists of finitely many points.
If $\mathcal{M}^\mu$ is diagonalisable with a single eigenvalue $1$ of multiplicity $4$, then $\mu$ is called a \emph{resonance point}.
When $\mu$ is a resonance point, every parallel section is an eigenvector of the monodromy matrix; thus, every Darboux transform is closed on the same period $T$ regardless of the choice of the initial condition.

Now let us further assume that $L$ admits a periodic polynomial conserved quantity $p^t$ sharing the same period as the polarised curve.
Then we have that $p^\mu$ is a periodic parallel section of $\dif{}^\mu$ implying $p^\mu(s) = A^\mu(s)p^\mu(0)$, and
    \[
        p^\mu(0) = p^\mu(T) = A^\mu(T)p^\mu(0) = \mathcal{M}^\mu p^\mu(0).
    \]
Hence, $p^\mu(0)$ is an eigenvector of the monodromy matrix $\mathcal{M}^\mu$ with corresponding eigenvalue $1$.

Diagonalisability of the monodromy matrix $\mathcal{M}^\mu$ implies the existence of two lightlike eigenvectors $\tilde{X}_1$ and $\tilde{X}_2$ with reciprocal real eigenvalues (see, for example, \cite[Chapter~III]{riesz_CliffordNumbersSpinors_1958}); therefore, let $\hat{L}_i := \langle \tilde{X}_i \rangle$ be the two closed Darboux transforms of $L$.
In particular, take $\hat{X}_i \in \Gamma \hat{L}_i$ for $i = 1, 2$ such that
    \[
        (\dif{} + \mu \eta)\hat{X}_i = 0,
    \]
and let $\alpha_i$ denote the corresponding eigenvalues, that is,
    \[
        \alpha_i \hat{X}_i(0) = \mathcal{M}^\mu \hat{X}_i(0).
    \]

If at least one of $\hat{X}_i$ is not perpendicular to $p^\mu$, that is, without loss of generality,
    \[
        (\hat{X}_1, p^\mu) \equiv c \neq 0
    \]
for some real constant $c$ due to Lemma~\ref{lem:par}, then we have using $\mathcal{M}^\mu \in \mathrm{SO}(3,1)$,
    \[
        \alpha_1 (\hat{X}_1(0), p^\mu(0)) = (\alpha_1 \hat{X}_1(0), p^\mu(T)) = (\mathcal{M}^\mu \hat{X}_1(0), \mathcal{M}^\mu p^\mu(0)) = (\hat{X}_1(0), p^\mu(0))
    \]
so that $\alpha_1 = 1$, and hence $\alpha_2 = \frac{1}{\alpha_1} = 1$.
Thus, if at least one of $\hat{X}_i$ is not perpendicular to $p^\mu$, then we have that
    \begin{equation}\label{eqn:eigenvec}
        p^\mu(0), \quad \hat{X}_1(0), \quad \hat{X}_2(0)
    \end{equation}
are all eigenvectors of $\mathcal{M}^\mu$, with eigenvalue $1$.

\textbf{Case (1):} If the three eigenvectors of $\mathcal{M}^\mu$ \eqref{eqn:eigenvec} are linearly independent, then the fourth eigenvalue must also equal to $1$ since $\det \mathcal{M}^\mu = 1$.
Thus, we have that $\mu$ is a resonance point.

\textbf{Case (2):} In contrast, if the three eigenvectors of $\mathcal{M}^\mu$ \eqref{eqn:eigenvec} are linearly dependent, then
    \[
        p^\mu(0) \in \langle \hat{X}_1(0), \hat{X}_2(0) \rangle =: \Phi \cong \mathbb{R}^{1,1},
    \]
so $\Phi$ has three fixed points under $\mathcal{M}^\mu$, and so $\mathcal{M}^\mu$ must act as the identity on $\Phi$.
Thus, viewing
    \[
        \mathbb{R}^{3,1} = \Phi \oplus \Phi^\perp \cong \mathbb{R}^{1,1} \oplus \mathbb{R}^2,
    \]
$\mathcal{M}^\mu$ must act as a rotation on $\Phi^\perp \cong \mathbb{R}^2$.
Therefore, the remaining two eigenvalues of $\mathcal{M}^\mu$ must be $\lambda^\mu, \bar{\lambda}^\mu$ for some $\lambda^\mu \in \mathbb{S}^1 \subset \mathbb{C}$ on $\Phi^\perp$.

\textbf{Case (2-a):} On one hand, if $\lambda^\mu$ is constant in $\mu$, then the real-analyticity of $\mathcal{M}^\mu$ implies that $\lambda^\mu$ and $\bar{\lambda}^\mu$ are roots to a second order polynomial equation with real-analytic coefficients. Thus, considering an appropriate analytic continuation of a square root term, it follows that $\lambda^\mu$ must be continuous in $\mu$.
However, since $\lambda^0 = 1$ as $\mathcal{M}^0 = id$, we must have that $\lambda^\mu \equiv 1$, and $\mu$ is a resonance point.

\textbf{Case (2-b):} On the other hand, suppose that $\lambda^\mu$ is continuous, but not constant in $\mu$.
Defining $\theta^\mu$ such that $e^{i \theta^\mu} := \lambda^\mu$, we note that since the set of rational numbers $\mathbb{Q}$ is dense over the set of real numbers, the non-constancy of $\theta^\mu$ tells us that there is some $\epsilon \in \mathbb{R}$ such that $\theta^{\mu + \epsilon} =: \theta^{\tilde{\mu}} \in \mathbb{Q}$.
For such $\tilde{\mu}$, there is some $\ell \in \mathbb{N}$ such that
    \[
        (\mathcal{M}^{\tilde{\mu}})^\ell = id,
    \]
so that $\tilde{\mu}$ is a resonance point over $\ell$-fold cover.

Summarizing all cases:
\begin{theorem}\label{thm:respcq}
    Let $L: (I, q) \to \mathbb{P}(\mathcal{L})$ be a closed polarised curve of finite type sharing the same period with a periodic polynomial conserved quantity $p^t$, and let $\mathcal{M}^\mu$ denote the monodromy matrix of Darboux transforms of $L$ with spectral parameter $\mu$.
    Let $\mu$ be chosen so that $\mathcal{M}^\mu$ is diagonalisable, and denote the two closed Darboux transforms of $L$ with spectral parameter $\mu$ by $\hat{L}_1, \hat{L}_2$.
    If we have
        \[
            (p^\mu, \hat{L}_1) \neq 0,
        \]
    then the eigenvalues of the monodromy matrix $\mathcal{M}^\mu$ must be
        \[
            1, \quad \lambda^\mu, \quad \bar{\lambda}^\mu
        \]
    for some $\lambda^\mu \in \mathbb{S}^1 \subset \mathbb{C}$.
    Furthermore, $L$ admits a resonance point on some multiple cover of $L$.
\end{theorem}

\section{Space form arc-length polarised curves}\label{sect:monolcq}
We now consider the integrable reduction to those curves polarised by space form arc-length polarisations.
\subsection{Curves polarised by space form arc-length polarisation}
First, we define space form arc-length polarisations in conformal geometry:
\begin{definition}
    Let $L : (I, q) \to \mathbb{P}(\mathcal{L})$ be a polarised curve, with $X \in \Gamma L$ such that $(X, \mathfrak{q}) = -1$ for some constant vector $
    \mathfrak{q}$.
    If we have
        \[
            q = (\dif{X}, \dif{X})
        \]
    then we call the polarisation a \emph{space form arc-length polarisation determined by $\mathfrak{q}$}.
    In particular, we will call the polarisation
        \begin{itemize}
            \item a \emph{spherical arc-length polarisation} if $(\mathfrak{q},\mathfrak{q}) < 0$,
            \item a \emph{Euclidean arc-length polarisation} if $(\mathfrak{q},\mathfrak{q}) = 0$, or
            \item a \emph{hyperbolic arc-length polarization} if $(\mathfrak{q},\mathfrak{q}) > 0$.
        \end{itemize}
\end{definition}
Otherwise said, we have a space form arc-length polarisation if the polarisation is arc-length with respect to the metric determined by some choice of $\mathfrak{q}$.
We will also say that the space form projection $x : (I, q) \to \mathbb{R}^2$ is polarised by space form arc-length if its lift $L : (I, q) \to \mathbb{P}(\mathcal{L})$ is polarised by space form arc-length.

\begin{example}
    Let $L : (I, q) \to \mathbb{P}(\mathcal{L})$ be polarised by space form arc-length determined by some $\check{\mathfrak{q}} \in \mathcal{L}$.
    As discussed in Section~\ref{sect:conf}, we use an auxiliary vector $\mathfrak{o} \in \mathcal{L}$ with $(\mathfrak{o},\check{\mathfrak{q}}) = -1$ to obtain a bijection $\check{\phi} : \check{M} \to \check{\mathfrak{R}}$.
    Denoting the Euclidean space form projection by $x := \check{\phi}(\check{X}) : (I, q) \to \mathbb{R}^2$, we then have by \eqref{eqn:0metric}
        \[
            q = (\dif{\check{X}}, \dif{\check{X}}) = |{\dif{x}}|^2,
        \]
    so that $x$ is polarised by (Euclidean) arc-length.
\end{example}

\begin{example}
    Now let $L : (I, q) \to \mathbb{P}(\mathcal{L})$ be polarised by space form arc-length determined by $\mathfrak{q}_\kappa$ as in \eqref{eqn:qkappa}.
    As discussed in Section~\ref{sect:conf}, we have a bijection $\phi_\kappa : M_\kappa \to \mathfrak{R}_\kappa$.
    Denoting by $x := \phi_\kappa(X_\kappa) : (I, q) \to \mathbb{R}^2$, we then have by \eqref{eqn:kappametric}
        \[
            q = (\dif{X}_\kappa, \dif{X}_\kappa) = \frac{4}{(1 + \kappa |x|^2)^2}|{\dif{x}}|^2.
        \]
    Thus, if $\kappa > 0$, then $x$ is polarised by spherical arc-length, while if $\kappa < 0$, then $x$ is polarised by hyperbolic arc-length, with the metric given by the stereoprojection.
\end{example}

\begin{example}
    Finally, let $L : (I, q) \to \mathbb{P}(\mathcal{L})$ be polarised by space form arc-length determined by $\tilde{\mathfrak{q}}$ as in Example~\ref{exam:halfplane}.
    Thus, using the bijection $\tilde{\phi} : \tilde{M} \to \tilde{\mathfrak{R}}$, we have $x = (x_1, x_2) := \tilde{\phi}(\tilde{X}) : (I, q) \to \mathbb{R}^2$, is polarised by
        \[
            q = (\dif{\tilde{X}}, \dif{\tilde{X}}) = \frac{1}{x_2^2}|{\dif{x}}|^2.
        \]
    Therefore, $x$ is polarised by hyperbolic arc-length, with the metric given by the Poincar\'{e} half-plane model.
\end{example}

In particular, we can characterise polarised curves with space form arc-length polarisations via the notion of polynomial conserved quantities of the associated family of flat connections.
\begin{theorem}\label{thm:lcq}
    Let $x : (I, q) \to \mathbb{R}^2$ be a polarised curve with its lift $L$.
    Then $x$ is polarised by space form arc-length determined by $\mathfrak{q}$ if and only if the lift $L$ of $x$ admits a linear conserved quantity of the form
        \begin{equation}\label{eqn:lcq}
            p^t = \mathfrak{q} + t X
        \end{equation}
    for $X \in \Gamma L$ taking values in $M := \{ X \in \mathcal{L} : (X, \mathfrak{q}) = -1 \}$.
\end{theorem}
\begin{proof}
    Let
        \[
            p^t = p_0 + t p_1 := \mathfrak{q} + t X,
        \]
    and write $\eta$ in terms of $ X \in \Gamma L$ such that $X$ takes values in $M$, that is,
        \[
            \eta = q \frac{X \wedge \dif{X}}{(\dif{X}, \dif{X})}.
        \]
    Now, note that $p_0 = \mathfrak{q}$ is constant, and $p_1$ is perpendicular to $\breve{X}$ and $\dif{\breve{X}}$ for any $\breve{X} \in \Gamma L$, while
        \begin{align*}
            0 = \dif{p_1} + \eta p_0 &= \dif{X} + \eta \mathfrak{q}\\
                &= \dif{X} + q \frac{X \wedge \dif{X}}{(\dif{X}, \dif{X})} \mathfrak{q}\\
                &= \dif{X} + \frac{q}{(\dif{X}, \dif{X})}((X, \mathfrak{q}) \dif{X} - (\dif{X}, \mathfrak{q}) X )\\
                &= \dif{X} - \frac{q}{(\dif{X}, \dif{X})}\dif{X}\\
                &= \left(1 - \frac{q}{(\dif{X}, \dif{X})}\right) \dif{X}
        \end{align*}
    if and only if
        \[
            q = (\dif{X}, \dif{X}).
        \]
    Thus, Lemma~\ref{lemma:pcqcond} allows us to conclude that a given polarised curve admits a linear conserved quantity of the form \eqref{eqn:lcq} if and only if it is polarised by a space form arc-length polarisation determined by $\mathfrak{q}$.
\end{proof}

\subsection{Monodromy of Darboux transforms of space form arc-length polarised curves}
Theorem~\ref{thm:lcq} implies that any curve polarised by space form arc-length polarisation admits a periodic linear conserved quantity.
Thus, the curve generically admits a resonance point by Theorem~\ref{thm:respcq}, where the spectral parameter $\mu$ must satisfy an additional condition that
    \[
        (p^\mu, \hat{X}_i) \neq 0
    \]
for one of two closed Darboux transforms $\hat{X}_1, \hat{X}_2$ of $X$.
    
\begin{figure}
	\centering
	\begin{minipage}{0.4\linewidth}
		\includegraphics[width=\textwidth]{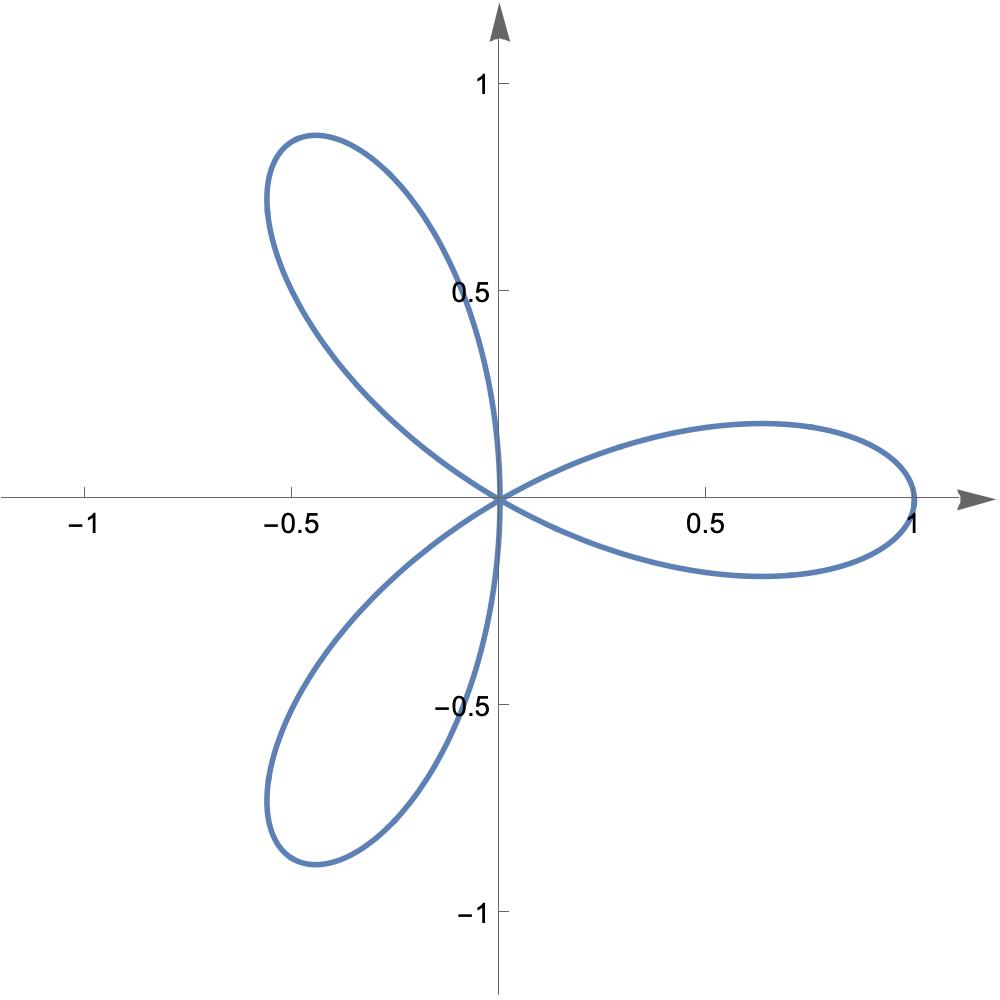}
		\[
			x(s) = \cos 3s (\cos s, \sin s)\]
	\end{minipage}
	\begin{minipage}{0.59\linewidth}
		\includegraphics[width=\textwidth]{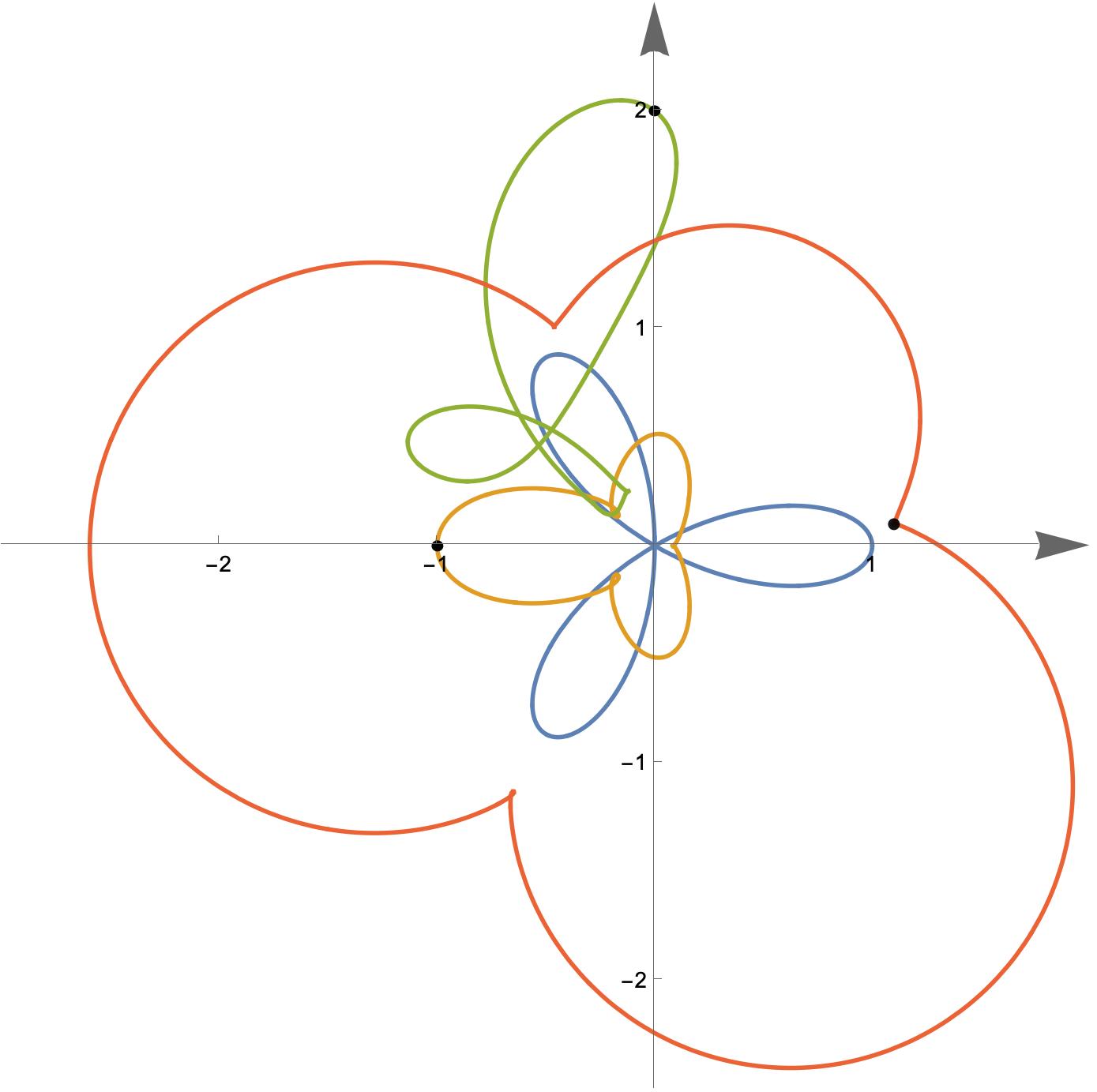}
	\end{minipage}
	\begin{minipage}{0.8\linewidth}
		\includegraphics[width=\textwidth]{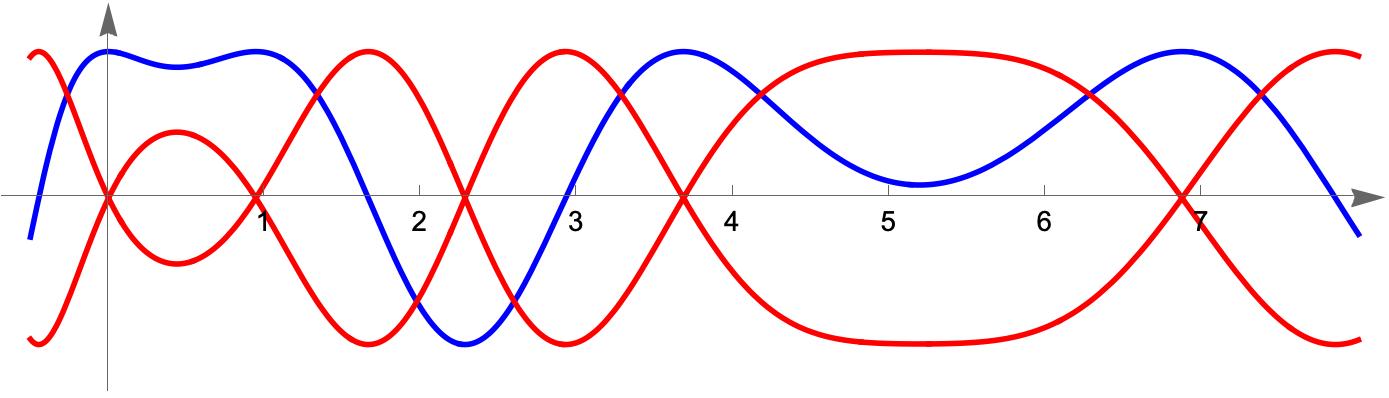}
	\end{minipage}
	\caption{Numerical example of Darboux transformations of a spherical arc-length polarised curve. The original curve (top left) is a rose curve given by the parametrisation stated; the real and imaginary parts of the complex conjugate pair of eigenvalues are plotted against the spectral parameter (on the bottom). The three Darboux transforms (top right) are closed over a single cover of the original curve, and are constructed using the resonance point $\mu \approx 0.94298625$.}
	\label{fig:exam1}
\end{figure}
	
\begin{figure}
	\centering
	\begin{minipage}{0.4\linewidth}
		\includegraphics[width=\textwidth]{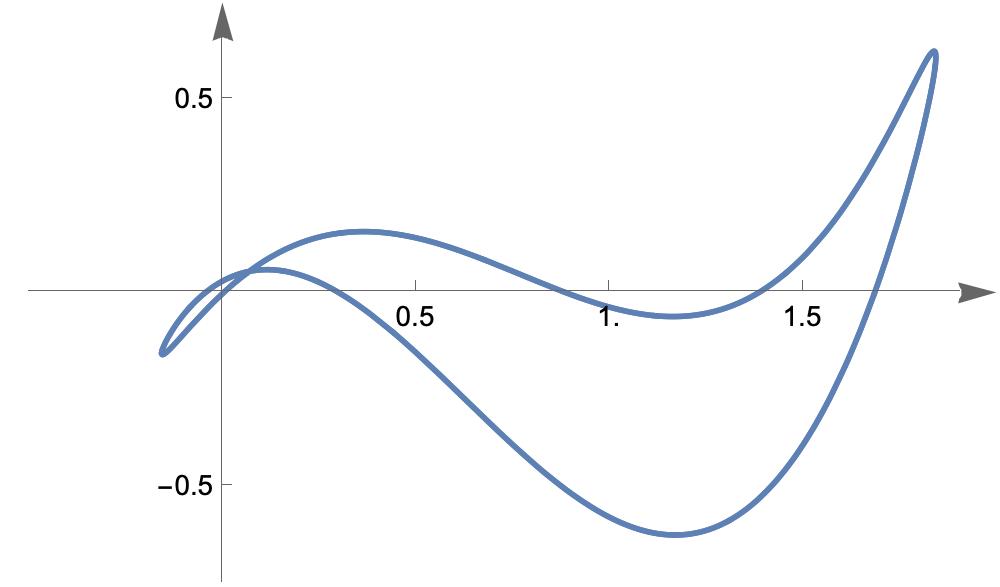}
		\begin{multline*}
			x(s) = (2\cos s \cdot \sin(s+1),\\
			 \sin 3s \cdot \cos 2s \cdot \cos s)\end{multline*}
	\end{minipage}
	\begin{minipage}{0.59\linewidth}
		\includegraphics[width=\textwidth]{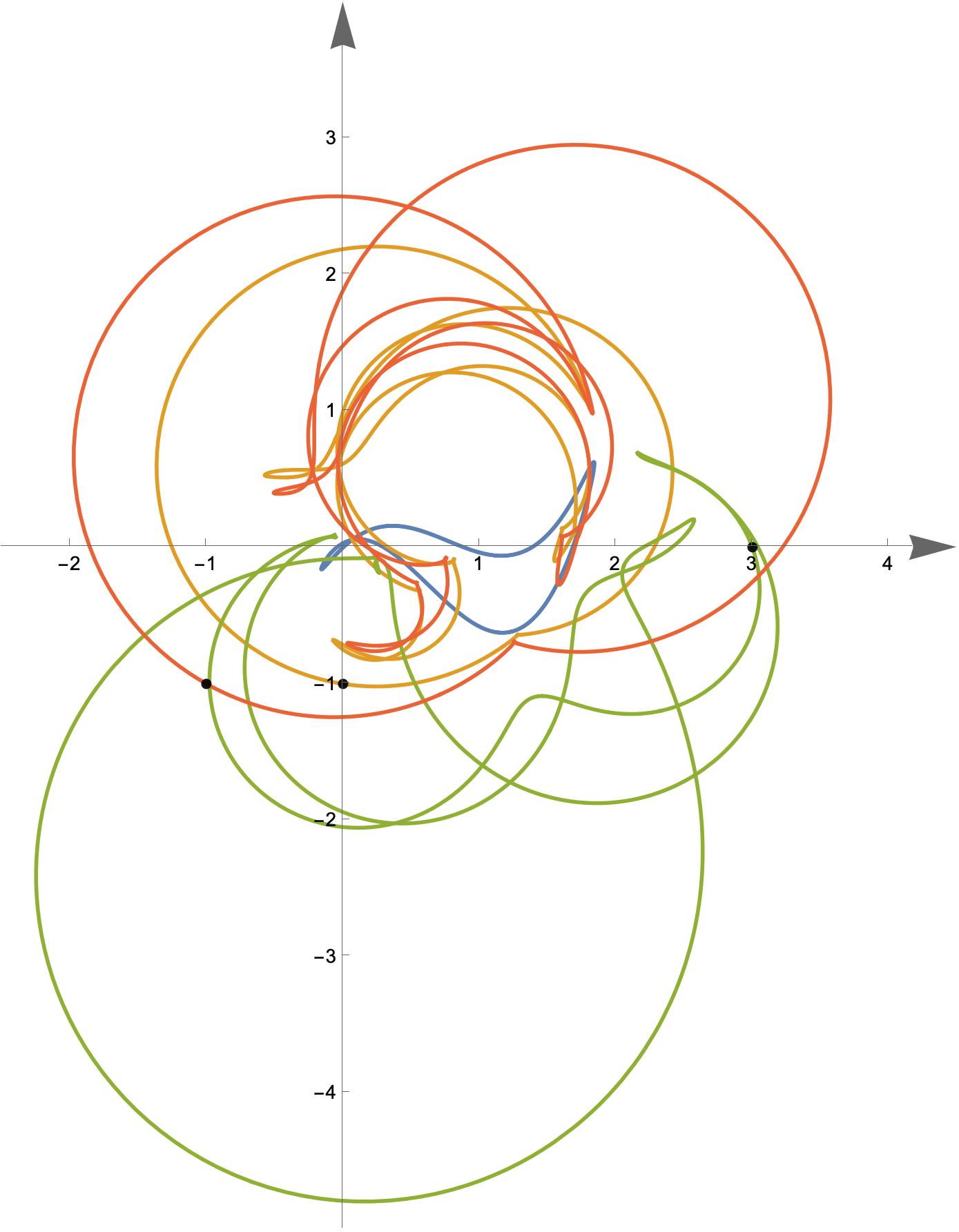}
	\end{minipage}
	\begin{minipage}{0.8\linewidth}
		\includegraphics[width=\textwidth]{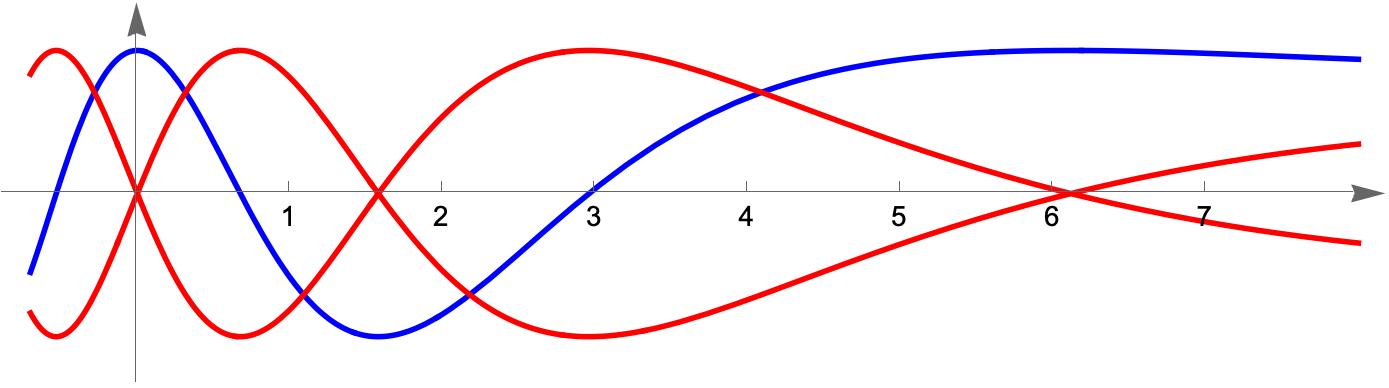}
	\end{minipage}
	\caption{Numerical example of Darboux transformations of a Euclidean arc-length polarised curve. The original curve (top left) is given by the parametrisation stated; the real and imaginary parts of the complex conjugate pair of eigenvalues are plotted against the spectral parameter (on the bottom). The three Darboux transforms (top right) are closed over a single cover of the original curve, and are constructed using the resonance point $\mu \approx 6.11654$.}
	\label{fig:exam2}
\end{figure}
To guarantee that at least one of closed Darboux transforms $\hat{X}_i$ is not perpendicular to $p^\mu$, assume without loss of generality that $L$ is a curve polarised by space form arc-length polarisation determined by some $\mathfrak{q}_\kappa$.
Then it admits a linear conserved quantity of the form
	\[
		p^t = \mathfrak{q}_\kappa + t X_\kappa,
	\]
and it follows that
    \begin{align*}
        (p^\mu,p^\mu)
            &= (\mathfrak{q}_\kappa + \mu X_\kappa, \mathfrak{q}_\kappa + \mu X_\kappa) \\
            &= (\mathfrak{q}_\kappa, \mathfrak{q}_\kappa) + 2\mu (\mathfrak{q}_\kappa, X_\kappa) + \mu^2 (X_\kappa, X_\kappa)
            = - \kappa - 2 \mu.
    \end{align*}
Thus, if we choose
    \[
        \mu > - \frac{\kappa}{2},
    \]
then we have $p^\mu$ is timelike throughout, so that lightlike $\hat{X}_i$ cannot be perpendicular to $p^\mu$.
Summarizing:
\begin{theorem}\label{thm:resonanceLCQ}
    Let $L: (I, q) \to \mathbb{P}(\mathcal{L})$ be a closed polarised curve of finite type with space form arc-length polarisation with respect to $\kappa$.
    Then $L$ admits a resonance point on some multiple cover of $L$ where the resonance point $\mu$ satisfies
        \[
            \mu > - \frac{\kappa}{2}.
        \]
\end{theorem}

For numerical examples, see Figures~\ref{fig:exam1} and \ref{fig:exam2}.

Since any constrained elastic curves, which are the critical points of the bending energy under constrained length and area, admit  space form arc-length polarisations \cite[Theorem~3.15]{cho_ConstrainedElasticCurves_2023}, we obtain the following corollary:
\begin{corollary}\label{coro:cec}
    Let $L: (I, q) \to \mathbb{P}(\mathcal{L})$ be a closed polarised curve of finite type that projects to a constrained elastic curve in some appropriate space form of sectional curvature $\kappa$.
    Then $L$ admits a resonance point on some multiple cover of $L$ where the resonance point $\mu$ satisfies
        \[
            \mu > -\frac{\kappa}{2}.
        \]
\end{corollary}

\begin{remark}
	We note here that every result regarding curves polarised by space form arc-length, also applies to curves polarised by \emph{space form negative arc-length polarisation},
	\[
		q = -(\dif{X}, \dif{X}),
	\]
	with appropriate changes in signs.
	Namely, they also admit
	\begin{itemize}
		\item a linear conserved quantity of certain form (cf.\ Theorem~\ref{thm:lcq}), and
		\item a resonance point in the appropriate range (cf.\ Theorem~\ref{thm:resonanceLCQ}),
	\end{itemize}

\end{remark}


\begin{bibdiv}
\begin{biblist}

\bib{bernstein_NonspecialNoncanalIsothermic_2001}{article}{
      author={Bernstein, Holly},
       title={Non-special, non-canal isothermic tori with spherical lines of
  curvature},
        date={2001},
     journal={Trans. Amer. Math. Soc.},
      volume={353},
      number={6},
       pages={2245\ndash 2274},
      review={\MR{1814069}},
      doi={10.1090/S0002-9947-00-02691-X}
}

\bib{bobenko_CompactBonnetPairs_2021}{article}{
      author={Bobenko, Alexander~I.},
      author={Hoffmann, Tim},
      author={{Sageman-Furnas}, Andrew~O.},
       title={Compact {{Bonnet Pairs}}: Isometric tori with the same
  curvatures},
      eprint={2110.06335},
      url={http://arxiv.org/abs/2110.06335}
}

\bib{burstall_SemidiscreteIsothermicSurfaces_2016}{article}{
      author={Burstall, Francis~E.},
      author={{Hertrich-Jeromin}, Udo},
      author={M{\"u}ller, Christian},
      author={Rossman, Wayne},
       title={Semi-discrete isothermic surfaces},
        date={2016},
     journal={Geom. Dedicata},
      volume={183},
       pages={43\ndash 58},
      review={\MR{3523116}},
      doi={10.1007/s10711-016-0143-7}
}

\bib{burstall_SpecialIsothermicSurfaces_2012}{article}{
      author={Burstall, Francis~E.},
      author={Santos, Susana~D.},
       title={Special isothermic surfaces of type $d$},
        date={2012},
     journal={J. Lond. Math. Soc. (2)},
      volume={85},
      number={2},
       pages={571–591},
      review={\MR{2901079}},
      doi={10.1112/jlms/jdr050}
}

\bib{cho_GeneralisedBianchiPermutability_2022}{article}{
      author={Cho, Joseph},
      author={Leschke, Katrin},
      author={Ogata, Yuta},
       title={Generalised {{Bianchi}} permutability for isothermic surfaces},
        date={2022},
     journal={Ann. Global Anal. Geom.},
      volume={61},
      number={4},
       pages={799\ndash 829},
      review={\MR{4423125}},
      doi={10.1007/s10455-022-09833-5}
}

\bib{cho_NewExplicitCMC_2022}{article}{
      author={Cho, Joseph},
      author={Leschke, Katrin},
      author={Ogata, Yuta},
       title={New explicit {{CMC}} cylinders and same-lobed {{CMC}}
  multibubbletons},
      eprint={2205.14675},
      url={http://arxiv.org/abs/2205.14675}
}

\bib{cho_ConstrainedElasticCurves_2023}{article}{
      author={Cho, Joseph},
      author={Pember, Mason},
      author={Szewieczek, Gudrun},
       title={Constrained elastic curves and surfaces with spherical curvature
  lines},
        date={2023},
     journal={Indiana Univ. Math. J.},
      volume={72},
      number={5},
       pages={2059\ndash 2099},
      review={\MR{4671894}},
      doi={10.1512/iumj.2023.72.9487}
}

\bib{cieslinski_IsothermicSurfaces$mathbf_1995}{article}{
      author={Cieśliński, Jan},
      author={Goldstein, Piotr},
      author={Sym, Antoni},
       title={Isothermic surfaces in $\mathbf E^3$ as soliton surfaces},
        date={1995},
     journal={Phys. Lett. A},
      volume={205},
      number={1},
       pages={37–43},
      review={\MR{1352426}},
      doi={10.1016/0375-9601(95)00504-V}
}

\bib{costa_ExampleCompleteMinimal_1984}{article}{
      author={Costa, Celso~J.},
       title={Example of a complete minimal immersion in $\mathbf{R}^3$ of genus
  one and three embedded ends},
        date={1984},
     journal={Bol. Soc. Brasil. Mat.},
      volume={15},
      number={1-2},
       pages={47–54},
      review={\MR{794728}},
      doi={10.1007/BF02584707}
}

\bib{darboux_SurfacesIsothermiques_1899}{article}{
      author={Darboux, Gaston},
       title={Sur les surfaces isothermiques},
        date={1899},
     journal={C. R. Acad. Sci. Paris},
      volume={128},
       pages={1299\ndash 1305},
}


\bib{hitchin_HarmonicMaps$2$torus_1990}{article}{
      author={Hitchin, N.~J.},
       title={Harmonic maps from a $2$-torus to the $3$-sphere},
        date={1990},
     journal={J. Differential Geom.},
      volume={31},
      number={3},
       pages={627–710},
      review={\MR{1053342}},
      doi={10.4310/jdg/1214444631}
}

\bib{hoffman_CompleteEmbeddedMinimal_1985}{article}{
      author={Hoffman, David~A.},
      author={Meeks, William~H., III},
       title={A complete embedded minimal surface in $\mathbf{R}^3$ with genus one
  and three ends},
        date={1985},
     journal={J. Differential Geom.},
      volume={21},
      number={1},
       pages={109–127},
      review={\MR{806705}},
      doi={10.4310/jdg/1214439467}
}

\bib{hoffman_EmbeddedMinimalSurfaces_1990}{article}{
      author={Hoffman, David~A.},
      author={Meeks, William~H., III},
       title={Embedded minimal surfaces of finite topology},
        date={1990},
     journal={Ann. of Math. (2)},
      volume={131},
      number={1},
       pages={1\ndash 34},
      review={\MR{1038356}},
      doi={10.2307/1971506}
}

\bib{riesz_CliffordNumbersSpinors_1958}{book}{
      author={Riesz, Marcel},
       title={Clifford numbers and spinors ({{Chapters I}}--{{IV}})},
      series={Lecture {{Series}}},
   publisher={{University of Maryland, Institute for Fluid Dynamics and Applied
  Mathematics}},
     address={College Park, MD},
        date={1958},
      volume={No. 38},
      review={\MR{181951}},
}

\bib{pinkall_ClassificationConstantMean_1989}{article}{
      author={Pinkall, Ulrich},
      author={Sterling, Ivan},
       title={On the classification of constant mean curvature tori},
        date={1989},
     journal={Ann. of Math. (2)},
      volume={130},
      number={2},
       pages={407\ndash 451},
      review={\MR{1014929}},
      doi={10.2307/1971425}
}

\bib{sterling_ExistenceClassificationConstant_1993}{article}{
      author={Sterling, Ivan},
      author={Wente, Henry~C.},
       title={Existence and classification of constant mean curvature
  multibubbletons of finite and infinite type},
        date={1993},
     journal={Indiana Univ. Math. J.},
      volume={42},
      number={4},
       pages={1239\ndash 1266},
      review={\MR{1266092}},
      doi={10.1512/iumj.1993.42.42057}
}

\bib{wente_CounterexampleConjectureHopf_1986}{article}{
      author={Wente, Henry~C.},
       title={Counterexample to a conjecture of {{H}}. {{Hopf}}},
        date={1986},
     journal={Pacific J. Math.},
      volume={121},
      number={1},
       pages={193\ndash 243},
      review={\MR{815044}},
      doi={10.2140/pjm.1986.121.193}
}

\end{biblist}
\end{bibdiv}

\end{document}